\documentclass[12pt,leqno]{article}

\usepackage{amssymb, graphicx, amscd, amsmath, amssymb, amsthm, 
array, fancyvrb, relsize, paralist, url}

\newtheorem{thm}{Theorem}[section]
\newtheorem{lem}[thm]{Lemma}
\newtheorem{cor}[thm]{Corollary}

\newtheorem*{rem}{Remark}

\numberwithin{equation}{section}

\def\F{\mathbb{F}_q}
\def\AA{\mathbb{A}^1}
\def\FF{\mathbb{F}_{q^2}}

\def\C{\mathcal{C}}

\def\B{\mathcal{B}}

\def\3F2{{}_3\hspace{-1pt}F_2}
\def\2F1{{}_2\hspace{-1pt}F_1}
\def\h2F1{{}_2\hspace{-1pt}\widehat{F}_1}
\def\f{\F^*}

\def\oalpha_c{\overline{\alpha_c}}

\def\l({\left(}
\def\r){\right)}
\def\bar{\begin{array}{r|}}
\def\ear{\end{array}}

\def\ord{\mathrm{ord}}
\def\mod{\ \mathrm{mod} \ }

\title{Irreducible factorization of translates of reversed Dickson 
polynomials over finite fields}

\author{\\ \\ Ron Evans\\
Department of Mathematics\\
University of California at San Diego\\
La Jolla, CA  92093-0112 \\
revans@ucsd.edu
\\ \\
and
\\ \\
Mark Van Veen\\
Varasco LLC \\
2138 Edinburg Avenue\\
Cardiff by the Sea, CA 92007  \\
mark@varasco.com
}

\medskip

\date{February 2018}

\begin{document}

\maketitle



\newpage

\begin{abstract}

Let $\F$ be a field of $q$ elements,
where $q$ is a power of an odd prime.
Fix $n = (q+1)/2$. For each $s \in \F$,
we describe all the irreducible factors over $\F$ of
the polynomial
$g_s(y): = y^n + (1-y)^n -s$, and we give a necessary and
sufficient condition on $s$ for $g_s(y)$ to be irreducible.
\end{abstract}

\section{Introduction}
Let $\F$ be a field of $q$ elements, 
where $q$ is a power of an odd prime $p$. 
Fix
\begin{equation}\label{eq:1.1}
n=(q+1)/2, 
\end{equation}
and define a polynomial $f(y) \in \F[y]$ of degree $[n/2]$ by
\[
f(y): = (1+\sqrt{y})^n + (1-\sqrt{y})^n = D_n(2,1-y),
\]
where $D_n(2,1-y)$ is a reversed Dickson polynomial \cite[eq.(1)]{HMSY}.
Our choice of $n$ in \eqref{eq:1.1}
was motivated by Katz's work on local systems
\cite{Katz}.
Indeed, by \cite[Lemma 2.1]{EV}, $f(y)$ satisfies the equality
\begin{equation}\label{eq:1.2}
f(y)^2 = 2y^n + 2(1-y)^n +2,
\end{equation}
which was instrumental in proving a theorem
of Katz relating two twisted local systems \cite[Theorem 16.8]{Katz}.

For each $s \in \F$, define the polynomial $g_s(y) \in \F[y]$ 
of degree $2[n/2]$ by 
\begin{equation}\label{eq:1.3}
g_s(y): = y^n + (1-y)^n -s = (f(y)^2 -2s-2)/2.
\end{equation}
Observe that
$g_s(y)$ is a translate
of the reversed Dickson polynomial
$g_0(y) =D_n(1,y-y^2)$ \cite[eq.(3)]{HMSY}.
For any zero $x$ of $g_s(y)$, \eqref{eq:1.3}
can be written as
\begin{equation}\label{eq:1.4}
g_s(y) = (f(y)^2 -f(x)^2)/2.
\end{equation}
By \eqref{eq:1.4} and \cite[Remark 2]{EV}, 
the zeros of $g_s(y)$ are all distinct
when $s \ne \pm 1$.

The goal of this paper is to describe the irreducible factorization
of $g_s(y)$ over $\F$, for each $s \in \F$.
We remark that irreducible factorizations of classical Dickson polynomials
over $\F$ have been given by Bhargava and Zieve \cite[Theorem 3]{BZ};
for related work, see the references in \cite[Section 9.6.2]{WY}.

Our study of the irreducible factors of $g_s(y)$ was initially motivated
by the following conjecture of the second author:

\noindent {\em For $s \in \{\pm 1/2\}$
and $q \equiv \pm 1 \pmod {12}$, every irreducible factor of 
$g_s(y)$  over $\F$ has the form 
$y^3 - (3/2)y^2 + (9/16)y - m$ 
for some $m \in \F$}.

\noindent For example, over $\mathbb{F}_{13}$, we have
the complete factorizations
\begin{equation}\label{eq:1.5}
\begin{split}
g_{-1/2}(y) &= y^7 +(1-y)^7 + 7=7(y^3+5y^2+3y+1)(y^3+5y^2+3y+3),\\
g_{1/2}(y)&= y^7 +(1-y)^7 - 7=7(y^3+5y^2+3y+6)(y^3+5y^2+3y+11).
\end{split}
\end{equation}
We found such formulas  intriguing,
as we initially saw no reason why the zeros of
$y^n + (1-y)^n \pm 1/2 $ should have degree 3 over $\F$ when
$q \equiv \pm 1 \pmod {12}$, nor did we understand why
all of the monic irreducible cubic 
factors over $\F$ should be identical except for
their constant terms.  

In Section 2, we present the irreducible factorizations of $g_s(y)$ 
corresponding to those $s$ for which the 
irreducible factors all have degree $\le 2$.
Before dealing with the more difficult case involving irreducible
factors of degree greater than $2$, 
we discuss properties of a 
second order linear recurring sequence $S$ in Section 3.
The sequence $S$ 
plays a crucial role in our proofs, although at first glance
it appears to have little to do with $g_s(y)$.

Our main results appear in Section 4.
Theorem 4.4 shows that if $1-s^2$ is a square $\ne 1$ in $\f$,
then every monic irreducible factor $I(y)$ of $g_s(y)$ has 
the same degree $d=e >2$, where $e$ is 
the period of the sequence $S$.  Theorem 4.5 then shows that
these factors $I(y)$ are all identical except for their constant
terms. Theorem 4.5 also gives formulas in terms of $s$
for the coefficients of the nonconstant terms of the factors $I(y)$; 
these formulas
are made explicit in Corollaries 4.6--4.12 
for some specific values of $s$, yielding all cases where  
$d$ is in the set $\{3,4,5,6,8,10,12\}$.
Corollary 4.6
in particular verifies the aforementioned conjecture of the second author.
Corollary 4.13 gives a necessary and sufficient condition 
on $s$ for the irreducibility of $g_s(y)$. 
We remark that Gao and Mullen \cite{GM} gave necessary and sufficient
conditions for the irreducibility of
translates of classical Dickson polynomials 
over $\F$.

Let $1-s^2$ be a square $\ne 1$ in $\f$.
Then the monic irreducible factors
$I(y)$ of $g_s(y)g_{-s}(y)$ 
are polynomials of degree $d>2$ that are all
identical  except for their constant terms.  
When $d$ is odd, Theorem 5.1 gives a
criterion that distinguishes the constant terms corresponding to $g_s(y)$
from those corresponding to $g_{-s}(y)$.
This will explain, for example, why the 
constant terms of the cubic factors in the first
equality of
\eqref{eq:1.5} had to be quadratic residues in $\mathbb{F}_{13}$, 
while those in the
second equality had to be quadratic non-residues.

\section{Linear and quadratic irreducible factors}

We first determine the irreducible factorization of $g_s(y)$ 
when $s \in \{\pm 1\}$.
Let $\rho$ denote the quadratic character on $\F$.
When $s=1$, \cite[Lemma 2.3]{EV} yields
\begin{equation}\label{eq:2.1}
2g_s(y) = 
f(y)^2-4 = \tau^2 y(y-1)\prod\limits_{a \in \C}(y-a)^2,
\end{equation}
where
$\tau$, the leading coefficient of $f$, is given by
\begin{equation}\label{eq:2.2}
\tau =
\begin{cases}
1, &\mbox{  if  } q \equiv 1 \mod 4 \\
2, &\mbox{  if  } q \equiv 3 \mod 4,
\end{cases}
\end{equation}
and the set $\C$ is defined by
\begin{equation}\label{eq:2.3}
\C: = \{a \in \F : \rho(a) = \rho(1-a) = 1\}.
\end{equation}
When $s=-1$, it follows from  \cite[Remark 3]{EV} that
\begin{equation}\label{eq:2.4}
2g_s(y) = f(y)^2 = \tau^2 \prod\limits_{j}(y-j)^2,
\end{equation}
where the product is over all $j \in \F$ for which
$\rho(j) = \rho(1-j) = -1$.
(The factor $\tau^2$ was inadvertently omitted in
\cite[Remark 3]{EV}.)
In summary, the irreducible factors of $g_s(y)$ 
have been completely determined when 
$s \in \{\pm 1\}$, and they are all linear.   

The case $s=0$ will be handled in Theorem 2.4.
The following theorem gives the 
irreducible factorization of $g_s(y)$ 
in the case that $g_s(y)$
has a zero in $\F$ and $s \notin \{0,\pm 1\}$.

\begin{thm}\label{Theorem 2.1}
Let $s \notin \{0, \pm 1\}$. 
The polynomial $g_s(y)$ has a zero in $\F$ if and only if
\begin{equation}\label{eq:2.5}
\rho((1+s)/2)=1, \quad  \rho((1-s)/2)=-1.
\end{equation}
When \eqref{eq:2.5} holds, $(1+s)/2$ and $(1-s)/2$ are the only zeros
in $\F$, and $g_s(y)$ has the irreducible factorization
\begin{equation}\label{eq:2.6}
\begin{split}
&g_s(y)= \\
&\frac{\tau^2}{2} \Big(y-\frac{1+s}{2}\Big)\Big(y-\frac{1-s}{2}\Big)
\prod\limits_{a \in \C}\Big(y^2 + (2as-1-s)y +\frac{(2a-1-s)^2}{4} \Big).
\end{split}
\end{equation}
\end{thm}

\begin{proof}
Observe that $g_s(y)$ has a zero $x \in \F$ if and only if
\begin{equation}\label{eq:2.7}
0 = g_s(x) = x\rho(x) +(1-x)\rho(1-x) - s,   \quad x \in \F,
\end{equation}
and since $s \notin \{\pm 1\}$,
\eqref{eq:2.7} is possible if and only if both $\rho(x) = -\rho(1-x)$
and $x \in \{(1 \pm s)/2\}$.
The equivalence of \eqref{eq:2.7} and
\eqref{eq:2.5} easily follows.

Suppose now that \eqref{eq:2.5} holds, so that $(1+s)/2$ and 
$(1-s)/2$ are the only zeros of $g_s(y)$ in $\F$.
Replacing $x$ by $(1+s)/2$ in the 
two equations \cite[(1.3),(1.4)]{EV} and then multiplying
these equations together, we obtain the factorization \eqref{eq:2.6}.

If one of the quadratic factors in  \eqref{eq:2.6} were reducible,
it would have a zero in $\F$.  Then, since $(1+s)/2$ and
$(1-s)/2$ are the only zeros of $g_s(y)$ in $\F$,
$g_s(y)$ would have a double zero in $\F$, contradicting the fact
that the zeros of $g_s(y)$ are all distinct when $s \ne \pm 1$.
This proves that all the factors in  \eqref{eq:2.6} are irreducible.
\end{proof}


The next two theorems 
deal with the cases where every irreducible factor $I(y)$ of $g_s(y)$ 
is quadratic. (This is in contrast to Theorem 2.1, where both
linear and quadratic factors $I(y)$ appeared.)
We first prove a lemma.

\begin{lem}\label{Lemma 2.2}
The set $ W=\{w \in \F: \rho((1+w)/2) = \rho((1-w)/2) = -1 \}$ 
has cardinality
$|W|=(q - \rho(-1))/4$.
\end{lem}

\begin{proof}
We have
\begin{equation}\label{eq:2.8}
|W|=\frac{1}{4} \sum\limits_{w \in \F}
\Big( 1 -\rho\Big( \frac{1+w}{2} \Big) \Big)
\Big( 1 -\rho\Big( \frac{1-w}{2} \Big) \Big)
=\frac{q - \rho(-1)}{4},
\end{equation}
where the last equality follows from \cite[Theorem 2.1.2]{BEW}.
\end{proof}

\begin{thm}\label{Theorem 2.3}
Suppose that  $s\notin \{ 0, \pm 1 \}$. Then
\begin{equation}\label{eq:2.9}
\rho((1+s)/2)=-1, \quad \rho((1-s)/2)= 1
\end{equation}
if and only if $g_s(y)$ has no zeros in $\F$ but has a zero $x$ of
degree 2 over $\F$.
When \eqref{eq:2.9} holds, $g_s(y)$ has the irreducible factorization
\begin{equation}\label{eq:2.10}
g_s(y) = \frac{\tau^2}{2} \prod\limits_{w} 
\Big(y^2-(1+sw)y+(s+w)^2/4\Big),
\end{equation}
where the product is over all $w \in \F$ for which
\begin{equation}\label{eq:2.11}
\rho((1+w)/2) = \rho((1-w)/2)) = -1.
\end{equation}
\end{thm}

\begin{proof}
Suppose that \eqref{eq:2.9} holds.   Then by Theorem 2.1, $g_s(y)$
has no zeros in $\F$.  We proceed to show that 
\begin{equation}\label{eq:2.12}
x:= \frac{1}{2} + \frac{sw}{2} +\frac{1}{2} \sqrt{1-w^2}\sqrt{1-s^2}
\end{equation}
is a zero of $g_s(y)$ for any $w \in \F$ satisfying \eqref{eq:2.11}.
We have the factorizations
\begin{equation}\label{eq:2.13}
x= \Big(\sqrt{(1+w)/2}\sqrt{(1+s)/2}
+\sqrt{(1-w)/2}\sqrt{(1-s)/2}\Big)^2
\end{equation}
and
\begin{equation}\label{eq:2.14}
1-x= 
\Big(\sqrt{(1+w)/2}\sqrt{(1-s)/2}
-\sqrt{(1-w)/2}\sqrt{(1+s)/2}\Big)^2,
\end{equation}
for appropriate choices of the square roots.
These two factorizations yield
\begin{equation}\label{eq:2.15}
\begin{split}
x^n+(1-x)^n &=\Big(\sqrt{(1+w)/2}\sqrt{(1+s)/2}
+\sqrt{(1-w)/2}\sqrt{(1-s)/2}\Big)^{q+1}\\
&+\Big(\sqrt{(1+w)/2}\sqrt{(1-s)/2}
-\sqrt{(1-w)/2}\sqrt{(1+s)/2}\Big)^{q+1}.
\end{split}
\end{equation}
Whenever $A, B \in \F$ and $\sqrt{B} \notin \F$, we have 
\[
(A+\sqrt{B})^{q+1} = (A+\sqrt{B})(A+\sqrt{B})^q =A^2 - B.  
\]
Thus 
when \eqref{eq:2.11} holds, the right member
of \eqref{eq:2.15} equals $s$.   This completes the proof
that $x$ is a zero of $g_s(y)$ whenever $w$ satisfies \eqref{eq:2.11}.

Conversely, suppose that
$g_s(y)$ has a zero of
degree 2 over $\F$ but has no zeros in $\F$.
We wish to prove \eqref{eq:2.9}.
Denote the zero of degree 2 by
\begin{equation}\label{eq:2.16}
x:=u + \sqrt{v}, \quad u,v \in \F, \quad \rho(v)=-1.
\end{equation}
Since $g_s(x)=0$,
we have $(s-x^n)^2 =(1-x)^{q+1}$, so that
\[
s^2 + u^2 -v-2sx^n = (1-u)^2 - v.
\]
Thus $2sx^n=s^2+2u-1$, and squaring gives
$4s^2(u^2-v) = (s^2 +2u -1)^2$.
Solving for $u$, we have
\begin{equation}\label{eq:2.17}
u = \frac{1}{2} + \frac{sw}{2}, \quad w:=\sqrt{1+4v/(s^2-1)}.
\end{equation}
(This definition of $w$
is consistent with \eqref{eq:2.11}, as will be shown shortly.)
Observe that $w=(2u-1)/s \in \F$.
Since $\rho(v) = -1$ and
\begin{equation}\label{eq:2.18}
(1-w^2)(1-s^2) = 4v,
\end{equation}
it follows that
\begin{equation}\label{eq:2.19}
\rho(1-w^2) = -\rho(1-s^2).
\end{equation}
In particular,
\begin{equation}\label{eq:2.20}
s^2 \ne w^2, \quad w^2 \ne 1.
\end{equation}
By \eqref{eq:2.16}--\eqref{eq:2.18},
\begin{equation}\label{eq:2.21}
x= \frac{1}{2} + \frac{sw}{2} +\frac{1}{2} \sqrt{1-w^2}\sqrt{1-s^2}.
\end{equation}
Thus by \eqref{eq:2.15},
\begin{equation}\label{eq:2.22}
\begin{split}
s&=\Big(\sqrt{(1+w)/2}\sqrt{(1+s)/2}
+\sqrt{(1-w)/2}\sqrt{(1-s)/2}\Big)^{q+1}\\
&+\Big(\sqrt{(1+w)/2}\sqrt{(1-s)/2}
-\sqrt{(1-w)/2}\sqrt{(1+s)/2}\Big)^{q+1}.
\end{split}
\end{equation}
Suppose for the purpose of contradiction that
$\rho((1+w)/2))$ $=$ $-\rho((1-w)/2)$,
so that by \eqref{eq:2.19},
$\rho((1+s)/2) = \rho((1-s)/2)$.
If 
$\rho((1+s)/2) = \rho((1+w)/2)$,
then 
\[
\sqrt{(1+w)/2}\sqrt{(1+s)/2} \in \F, \quad
\sqrt{(1+w)/2}\sqrt{(1-s)/2} \in \F
\]
and
\[
\sqrt{(1-w)/2}\sqrt{(1+s)/2} \notin \F, \quad
\sqrt{(1-w)/2}\sqrt{(1-s)/2} \notin \F,
\]
so by \eqref{eq:2.22},
$s=w$, which contradicts \eqref{eq:2.20}.
Similarly, if
$\rho((1+s)/2) = \rho((1-w)/2)$,
we obtain the contradiction $s = -w$.
This contradiction shows that
\begin{equation}\label{eq:2.23}
\rho((1+w)/2) = \rho((1-w)/2).
\end{equation}
Then by \eqref{eq:2.19},
\begin{equation}\label{eq:2.24}
\rho(1-s^2) = -1.
\end{equation}
By Theorem 2.1,  \eqref{eq:2.5} cannot hold, so \eqref{eq:2.24} yields
\eqref{eq:2.9}, as desired.

Next we show that $w$ satisfies \eqref{eq:2.11}.
Suppose for the purpose of contradiction that the two members of
\eqref{eq:2.23} are equal to $1$. Then by \eqref{eq:2.9},
\[
\sqrt{(1+w)/2}\sqrt{(1+s)/2} \notin \F, \quad
\sqrt{(1-w)/2}\sqrt{(1+s)/2} \notin \F,
\]
and
\[
\sqrt{(1-w)/2}\sqrt{(1-s)/2} \in \F, \quad
\sqrt{(1+w)/2}\sqrt{(1-s)/2} \in \F.
\]
Then  \eqref{eq:2.22} yields the
contradiction $s = -s$.  This contradiction shows that
$w$ satisfies \eqref{eq:2.11}.

Assuming now \eqref{eq:2.9}, we have only to prove the irreducible
factorization in \eqref{eq:2.10}.   In view of \eqref{eq:2.12},
\[
(y-x)(y-x^q) = y^2 - (1+sw)y +(s+w)^2/4
\]
is an irreducible factor of $g_s(y)$ over $\F$,
for each $w \in \F$ satisfying \eqref{eq:2.11}.
Since the degree of $g_s(y)$ is equal to $(q - \rho(-1))/2$, it
remains to show that there are $(q - \rho(-1))/4$ choices of $w \in \F$
for which \eqref{eq:2.11} holds.   This follows from Lemma 2.2.
\end{proof}

\begin{thm}\label{Theorem 2.4}
Let $s=0$. Then
$g_s(y)$ has the irreducible factorization
\begin{equation}\label{eq:2.25}
g_s(y) = \frac{\tau^2}{2} \prod\limits_{v} (y^2-y+v/4),
\end{equation}
where the product is over all $v \in \F$ for which
\begin{equation}\label{eq:2.26}
\rho(v) = \rho(1-v) = -1.
\end{equation}
\end{thm}

\begin{proof}
Suppose that \eqref{eq:2.26} holds and
\[
x:=(1+\sqrt{1-v})/2, \quad 1-x=(1-\sqrt{1-v})/2.
\]
Then
\begin{equation*}
\begin{split}
(1-&\sqrt{1-v})^n \   2^n (x^n + (1-x)^n) \\
&=(1-\sqrt{1-v})^n ( (1+\sqrt{1-v})^n  + (1-\sqrt{1-v})^n) \\
&= v^n + (1-\sqrt{1-v})^{q+1}
= v^{(q+1)/2}  + v =v (\rho(v) + 1) = 0.
\end{split}
\end{equation*}
Therefore $x$ is a zero of $g_0(y)$, so that
\[
(y-x)(y-x^q) = y^2-y+v/4
\]
is an irreducible factor of $g_0(y)$ over $\F$.
Since the degree of $g_0(y)$ is equal to $(q - \rho(-1))/2$, it
remains to show that there are $(q - \rho(-1))/4$ choices of $v \in \F$
for which \eqref{eq:2.26} holds.   This follows by comparing degrees
on both sides of \eqref{eq:2.4}.  (Alternatively, it can be deduced from
Lemma 2.2.)
\end{proof}

When $s=0$, Theorem 2.4 shows that the irreducible factors
are all quadratic.  When $s \in \{ \pm 1 \}$,
\eqref{eq:2.1} and \eqref{eq:2.4} show that the irreducible factors
are all linear.  When $s \notin \{0, \pm 1 \}$,
Theorems 2.1 and 2.3 show that                  
the irreducible factors of $g_s(y)$
all have degree $\le 2$ if and only if $\rho(1-s^2) = -1$.
In each case above, the irreducible factorization is completely determined.
In Section 4, we consider those remaining $s$ for which
\begin{equation}\label{eq:2.27}
\rho(1 - s^2) = 1, \quad s \ne 0.
\end{equation}
These are precisely the values of $s$ for which $g_s(y)$ has an
irreducible factor of degree $>2$.

\section{A second order linear recurrence sequence}

Define  $c=1-s^2$.  From here on, we will always assume that
\begin{equation}\label{eq:3.1}
\rho(c) = \rho(1-c) = 1, \ \mbox{   i.e.,   } \quad 1-s^2=c \in \C.
\end{equation}
This is just a restatement of \eqref{eq:2.27}.

For $-\infty < k < \infty$, define a bilateral 
second order linear recurrence sequence 
$S(c):=\langle c_k \rangle$  in $\F$ by
\begin{equation}\label{eq:3.2}
c_{k+1} = (2-4c)c_k -c_{k-1} +2c, \quad c_0=0, \ c_1=c.
\end{equation}
For example, 
$c_2 = -4c^2 +4c$, $\ c_3 = 16c^3 -24c^2 +9c$,
and for a general positive integer $k$, $c_k=c_{-k}$ 
equals a polynomial in $c$
over the integers with leading term $(-4)^{k-1} c^k$.
If each of the three $c_i$'s in \eqref{eq:3.2} is replaced by $c_i +1/2$,
then the inhomogeneous sequence in \eqref{eq:3.2} is replaced
by a homogeneous one.
The characteristic polynomial corresponding to the homogeneous sequence is
\begin{equation}\label{eq:3.3}
y^2 + (4c-2)y + 1.
\end{equation}
Let $i \in \FF$ denote a fixed square root of $-1$.
The zeros of the polynomial in \eqref{eq:3.3} are $\beta^2$ and $\beta^{-2}$,
where
\begin{equation}\label{eq:3.4}
\beta: = \sqrt{1-c} + i \sqrt{c}, \quad
\beta^2 = 1-2c + 2i\sqrt{1-c}\sqrt{c},
\end{equation}
so that
\[
\beta^{-1} = \sqrt{1-c} - i \sqrt{c}, \quad
\beta^{-2} = 1-2c - 2i\sqrt{1-c}\sqrt{c}.
\]
By \eqref{eq:3.1}, $\beta \in \F[i]$,
and $\beta \in \F$ if and only if $\rho(-1)=1$.
Using the well known evaluation of homogeneous linear recurrence sequences
\cite[10.2.17]{HN}, we obtain the closed form evaluations
\begin{equation}\label{eq:3.5}
c_k = \frac{-1}{4}(\beta^k - \beta^{-k})^2
\end{equation}
for every integer $k$.
A direct calculation using \eqref{eq:3.5} yields the (nonlinear)
recurrence relation
\begin{equation}\label{eq:3.6}
c_{k+1} c_{k-1} = (c - c_k)^2.
\end{equation}

Note that the right member of \eqref{eq:3.5} does not depend
on which of the ambiguous signs of the square roots
are chosen in the formula for $\beta$ in \eqref{eq:3.4}.
Fix one of the choices of $\beta$. We may assume that $\beta$
has even order, since otherwise we could replace $\beta$ by $-\beta$.

By  \cite[10.2.4]{HN}, the sequence $S(c)$ is purely periodic.
Let $e$ denote its period.
Thus for each integer $k$,
\begin{equation}\label{eq:3.7}
c_{e+k}=c_k = c_{-k} = c_{e-k}.
\end{equation}
Write $\theta=\ord (\beta^2)$ (the order of $\beta^2$).
Thus $\ord (\beta) = 2\theta$.
By \eqref{eq:3.5}, $c_k = c_{k + \theta}$ for all $k$,
so that $e$ divides $ \theta$.
Also by \eqref{eq:3.5}, 
$c_k = 0$ if and only if $\theta$ divides $k$.
Since $c_e = c_0 = 0$ by \eqref{eq:3.7},  
$\theta$ divides $e$, so $\theta = e$. In summary,
\begin{equation}\label{eq:3.8}
\ord (\beta) =2e
\end{equation}
and
\begin{equation}\label{eq:3.9}
c_k = 0 \ \mbox{  if and only if } \ e|k.
\end{equation}

By \eqref{eq:3.5},
\begin{equation}\label{eq:3.10}
1-c_k = \frac{1}{4}(\beta^k + \beta^{-k})^2,
\quad s=\pm \sqrt{1-c_1} = \pm ( \beta + \beta^{-1})/2.
\end{equation}
Whether or not $\beta \in \F$, it follows
from \eqref{eq:3.4}, \eqref{eq:3.5} and \eqref{eq:3.10} that
\begin{equation}\label{eq:3.11}
\rho(c_k) = \rho(1-c_k) = 1, \ \mbox{   i.e.,} \quad c_k \in \C,
\ \ \mbox{whenever} \quad c_k \notin \{0,1\}.
\end{equation}
From \eqref{eq:3.10}, we also see that 
\begin{equation}\label{eq:3.12}
c_k = 1 \ \mbox{ if and only if  }  2|e \ \mbox{ and }  k \equiv e/2 \pmod e.
\end{equation}

For $0 \le j < k \le e/2$, we claim that $c_j \ne c_k$.
To see this, suppose otherwise.  Then by \eqref{eq:3.10},
$\beta^k + \beta^{-k} = \epsilon (\beta^j + \beta^{-j})$
where $\epsilon \in \{\pm 1\}$.
Equivalently,
$\beta^k (1 - \epsilon \beta^{j-k}) 
=\epsilon \beta^{-j} (1 - \epsilon \beta^{j-k})$.
The factors in parentheses are nonzero, so they can be canceled to yield
$\beta^{j+k} = \epsilon$.
This is impossible, since $j+k$ lies strictly between 0 and $e$.
Thus the claim is proved.

If  $\beta \in \F$, then $\beta^{q-1}=1$.
If  $\beta \notin \F$, then $\beta^q = \beta^{-1}$ by \eqref{eq:3.4},
so that $\beta^{q+1}=1$.
Therefore  in all cases, it follows from \eqref{eq:3.8} that
$e$ divides $E$, where $E$ is the even integer defined by
\begin{equation}\label{eq:3.13}
E:=2[n/2]=(q - \rho(-1))/2.
\end{equation}

If $c = -(\zeta -\zeta^{-1})^2/4$ for some 
$\zeta \in \F[i]$ of order $2e$ with $e \mid E$,
then $S(c)$ has period $e$ and \eqref{eq:3.11} holds.
For example, suppose that $B$ is an element of $\F[i]$
of full order $2E$.  (In the case $q \equiv 1 \pmod 4$,
this means that $B$ is a primitive root in $\F$.) 
A special case of the general sequence $S(c)$ of period $e$ is the sequence
$S(C)=\langle C_k \rangle$ of period $E$, where (cf. \eqref{eq:3.5}),
\begin{equation}\label{eq:3.14}
C_k: = \frac{-1}{4}(B^k - B^{-k})^2, \quad C:=C_1 =\frac{-1}{4}(B - B^{-1})^2.
\end{equation}
We have $C_0=0$ and $C_{E/2} = 1$, and by \eqref{eq:3.11},
the set $\{ C_j: 1 \le j \le E/2 -1\}$ is a subset of $\C$ of cardinality
$E/2 - 1$. But $\C$ itself has cardinality $E/2 - 1$, 
which can be seen by comparing the degrees on both sides of \eqref{eq:2.1}.
Therefore,
\begin{equation}\label{eq:3.15}
\{ C_j: 1 \le j \le E/2 -1\} = \C.
\end{equation}

For $u = \pm 1$, define, for each integer $k$,
\begin{equation}\label{eq:3.16}
A_u(k,c_1) = c_k+c_1-2c_kc_1+2u\sqrt{c_k-c_k^2}\sqrt{c_1-c_1^2}
\end{equation}
and
\begin{equation}\label{eq:3.17}
A'_u(k,c_1) = c_k+c_1-2c_kc_1-2u\sqrt{c_k-c_k^2}\sqrt{c_1-c_1^2},
\end{equation}
so that $A'_u = A_{-u}$.
We drop the subscript $u$ when $u=1$.
Define the set
\begin{equation}\label{eq:3.18}
Z(k,c_1): = \{ A(k,c_1), A'(k,c_1)\}.
\end{equation}

By \eqref{eq:3.11}, the functions $A(k,c_1), A'(k, c_1)$
have values in $\F$.  We can make these functions
single-valued for each $k$ by specifying
the signs of the square roots of 
$c_k$, $1-c_k$, and $c_k-c_k^2$ 
in terms of our fixed $\beta$,
as follows:
\begin{equation}\label{eq:3.19}
\sqrt{c_k}: = \frac{-i}{2}(\beta^k - \beta^{-k}), \quad
\sqrt{1-c_k}: = \frac{1}{2}(\beta^k + \beta^{-k}),
\end{equation}
and
\begin{equation}\label{eq:3.20}
\sqrt{c_k-c_k^2}: =\sqrt{c_k}\sqrt{1-c_k}= 
\frac{-i}{4}(\beta^{2k} - \beta^{-2k}).
\end{equation}
Note that the values of these square roots depend not just on their arguments
but on the subscripts  $k$ as well.
For example, $\sqrt{1-c_1} \ne \sqrt{1-c_{e-1}}$ when $e>2$,
even though $c_1 = c_{e-1}$.

\begin{lem}\label{Lemma 3.1}
For each integer $k$,
\begin{equation}\label{eq:3.21}
Z(k,c_1) = \{c_{k-1}, c_{k+1}\}.
\end{equation}
\end{lem}

\begin{proof}
By \eqref{eq:3.16}--\eqref{eq:3.18}, $Z(k, c_1)$
consists of the two elements
\[
c_k+c_1-2c_kc_1 \pm 2\sqrt{c_k-c_k^2}\sqrt{c_1-c_1^2}.
\]
Express each of these two elements as a sum of powers of $\beta$ using
\eqref{eq:3.5} and \eqref{eq:3.20}.   
A longish computation (facilitated by a computer algebra
program) then shows that
these two elements reduce to  $c_{k-1}$ and $c_{k+1}$.
\end{proof}

\section{Irreducible factors of degree exceeding 2}

Recall from \eqref{eq:3.1}  
that
$1-s^2 = c=c_1  \in \C$,
so that $g_s(y)$ has at least one zero whose degree 
(over $\F$) exceeds 2.
It will be shown below \eqref{eq:4.10} that all the zeros have
the same degree.  Theorem 4.4 will show that the common degree
is $e$, where $e>2$ is the period of the sequence $S(c)$.

Given any zero $x$ of $g_s(y)$, the
following lemma gives a formula for $x^q$ in terms of $c_1$ and $x$.
\begin{lem}\label{Lemma 4.1}
Let $x$ be a zero of $g_s(y)$. Then for some $v = \pm 1$,
\[
x^q = c_1 + x -2c_1x + 2v \sqrt{c_1-c_1^2}\ \sqrt{x-x^2}.
\]
\end{lem}

\begin{proof}
Since
$(1-x)^n = s-x^n$,
squaring yields
\[
(1-x)^{q+1} = x^{q+1} + s^2 - 2sx^n.
\]
Thus 
$1 - x - x^q = s^2 - 2sx^n$,
which is equivalent to
\begin{equation}\label{eq:4.1}
x^q = 1 - s^2 -x +2sx^n.
\end{equation}
Multiplication by $x$ gives
\[
x^{q+1} = (1-s^2)(x-x^2) +2sx^{n+1} - s^2x^2, 
\]
which simplifies to
\begin{equation}\label{eq:4.2}
(x^n - sx)^2 = (1-s^2) (x-x^2).
\end{equation}
Note that this shows that $x-x^2$ is a square in $\F [x]$.
By \eqref{eq:4.1},
\[
x^q = 1-s^2 +x(2s^2-1) +2sx^n - 2s^2x.
\]
Then applying \eqref{eq:4.2}, we obtain
\[
x^q = 1-s^2 + x(2s^2-1) \pm 2s \sqrt{1-s^2}\sqrt{x-x^2},
\]
and since $c_1=1-s^2$, this yields the desired formula
\[
x^q = c_1 + x -2c_1x \pm 2 \sqrt{c_1-c_1^2} \ \sqrt{x-x^2}.
\]
\end{proof}

We proceed to extend the definitions in \eqref{eq:3.16}--\eqref{eq:3.18}.
For $u = \pm 1$, define, for each 
integer $k$ and each zero $x$ of $g_s(y)$,
\begin{equation}\label{eq:4.3}
A_u(k,x) = c_k+x-2c_kx+2u\sqrt{c_k-c_k^2} \ \sqrt{x-x^2}
\end{equation}
and
\begin{equation}\label{eq:4.4}
A'_u(k,x) = c_k+x-2c_kx-2u\sqrt{c_k-c_k^2} \ \sqrt{x-x^2}.
\end{equation}
We drop the subscript $u$ when $u=1$.
Define the set
\begin{equation}\label{eq:4.5}
Z(k,x): = \{ A(k,x), A'(k,x)\}.
\end{equation}
In this notation, Lemma 4.1 states that
\begin{equation}\label{eq:4.6}
x^q \in Z(1,x).
\end{equation}

The values of the functions $A(k,x), A'(k,x)$ lie in $\F[x]$.
We can make these functions
single-valued for each 
zero $x$ of $g_s(y)$, by employing  \eqref{eq:4.2} 
to fix a choice of sign of $\sqrt{x-x^2}$, as follows:
\begin{equation}\label{eq:4.7}
\sqrt{x-x^2} = (x^n - sx)/\sqrt{c_1},
\end{equation}
where $\sqrt{c_1}$ is specified in \eqref{eq:3.19}.
In \eqref{eq:4.7},when the zero $x$ is replaced by the zero $1-x$,
the argument of the square root on the left remains the same,
but the sign of the square root is changed.
Also, when the zero $x$ is replaced by its conjugate zero $x^q$,
we see that
\begin{equation}\label{eq:4.8}
\sqrt{x^q-x^{2q}}=(\sqrt{x-x^2})^q.
\end{equation}

Observe that
\begin{equation}\label{eq:4.9}
\sqrt{A_u(k,x) - A_u(k,x)^2} = 
\pm \Big((2x-1)\sqrt{c_k-c_k^2} +u(2c_k-1)\sqrt{x-x^2} \Big),
\end{equation}
and the same formula holds when each $x$ is replaced by $c_1$.
This is readily verified upon squaring both sides.

Fix a zero $x$ of $g_s(y)$ of smallest degree $d$ over $\F$.
In view of \eqref{eq:1.4}, it follows from \cite[Theorem 3.1]{EV}
that
\begin{equation}\label{eq:4.10}
\{a+x-2ax \pm 2\sqrt{a-a^2}\sqrt{x-x^2}: a \in \{0,1\} \cup \C \}
\end{equation}
is the set of zeros of $g_s(y)$.  Since
$x-x^2$ is a square in $\F [x]$, these zeros all lie in $\F [x]$,
so their degrees cannot exceed $d$.  
Then by minimality of $d$,
every zero $x$ of $g_s(y)$ has degree $d$.
Since at least one zero has degree exceeding 2, we conclude that
$d > 2$.

The next lemma shows how the set $Z(k,x^q)$ depends on  $c_{k-1}$
and $c_{k+1}$.

\begin{lem}\label{Lemma 4.2}
Let $x$ be a zero of $g_s(y)$ and let $k$ be an integer.   Then
for some $\mu, \lambda \in \{\pm 1\}$,
\[
Z(k, x^q) = \{A_\mu (k-1,x), A_\lambda (k+1,x)\}.
\]
\end{lem}

\begin{proof}
For each $ u \in \{\pm 1\}$,
\begin{equation}\label{eq:4.11}
A_u (k, x^q) = c_k +(1-2c_k) x^q +2u \sqrt{c_k-c_k^2}  \ \sqrt{x^q -x^{2q}}.
\end{equation}
By \eqref{eq:4.6}, we can replace $x^q$ by $A_v(1,x)$ for
some $v \in \{ \pm 1 \}$, after which
we can apply \eqref{eq:4.9} to the rightmost square root in 
\eqref{eq:4.11} to obtain
\begin{equation}\label{eq:4.12}
\begin{split}
&A_u (k, x^q) =  c_1+c_k - 2c_1c_k + \\ 
& + x(1-2c_1-2c_k+4c_1c_k) 
+2v(1-2c_k)\sqrt{c_1-c_1^2} \ \sqrt{x-x^2} \\
& -2w\sqrt{c_k-c_k^2} \  \Big(v(2c_1-1)\sqrt{x-x^2} + (2x-1)\sqrt{c_1-c_1^2} \ \Big),
\end{split}
\end{equation}
where $w \in \{\pm 1\}$ depends on $u$.
Once again applying \eqref{eq:4.9}, we see that \eqref{eq:4.12} reduces to
\begin{equation}\label{eq:4.13}
A_u (k, x^q) =x+(1-2x)A_w(k,c_1)
\pm 2\sqrt{x-x^2}\sqrt{A_w(k,c_1) - A_w(k,c_1)^2}.
\end{equation}
Repeating the entire argument above with $-u$ in place of $u$,
we see that \eqref{eq:4.13} holds with the signs of $u$ and $w$ reversed, i.e.,
\begin{equation}\label{eq:4.14}
A'_u (k, x^q) =x+(1-2x)A'_w(k,c_1)
\pm 2\sqrt{x-x^2}\sqrt{A'_w(k,c_1) - A'_w(k,c_1)^2}.
\end{equation}
By Lemma 3.1, there exists $\epsilon \in \{ \pm 1 \}$ for which
\[
A_w(k,c_1) = c_{k-\epsilon}, \quad A'_w(k,c_1) = c_{k+\epsilon}.
\]
Therefore \eqref{eq:4.13} and \eqref{eq:4.14} yield
\[
A_u (k, x^q)  \in Z(k-\epsilon,x), \quad
A'_u (k, x^q)  \in Z(k+\epsilon,x).
\]
Since $A_u (k, x^q)$ and $A'_u (k, x^q)$
are the two elements of the set $Z(k, x^q)$,
the lemma is proved.
\end{proof}

For ease in notation, write
$x_k: = x^{q^k}$ for the conjugates of $x$.
The following crucial theorem 
shows that for any integer $k$ with $0 \le k \le d$,
the set $\{ x_k, x_{d-k} \}$ is equal to the set
 $\{ A(k,x), A'(k,x)\}$.
\begin{thm}\label{Theorem 4.3}
Let $1-s^2 = c = c_1 \in \C$ and let $x$ be a zero of $g_s(y)$ of degree $d >2$
over $\F$.  Then
\begin{equation}\label{eq:4.15}
Z(k,x) = \{ x^{q^k}, x^{q^{d-k}} \}, \quad k=0,1,\dots, d.
\end{equation}
\end{thm}

\begin{proof}
Since $c_0 = 0$ and $x_0 = x_d =x$, \eqref{eq:4.15} holds for $k=0$.
Next we prove \eqref{eq:4.15} for $k=1$.
Fix $v$ as in Lemma 4.1, so that $x^q=A_v(1,x)$. Then
\[
Z(1,x^q) = \{A_v(1,x)^q, A'_v(1,x)^q\} 
=\{x_2, A'_v(1,x)^q\},
\]
where the first equality follows from \eqref{eq:4.8}.
On the other hand, by Lemma 4.2 with $k=1$,
\[
Z(1,x^q) = \{x, A_\lambda(2,x)\}.
\]
Since $x$ has degree $d > 2$, we cannot have $x = x_2$.
Therefore $x = A'_v(1,x)^q$.  Raising both sides to the power
$q^{d-1}$, we see that $x_{d-1} = A'_v(1,x)$.
Consequently,
\[
Z(1,x) = \{A_v(1,x), A'_v(1,x)\} = \{x_1, x_{d-1}\},
\]
which proves \eqref{eq:4.15} for $k=1$.

Now let $1 \le k < d$ and assume as induction hypothesis that
\begin{equation}\label{eq:4.16}
Z(j,x) = \{ x_j, x_{d-j} \}, \quad  0 \le j \le k.
\end{equation}
We need to prove
\begin{equation}\label{eq:4.17}
Z(k+1,x) = \{ x_{k+1}, x_{d-k-1} \}.
\end{equation}
We will first prove \eqref{eq:4.17}
when $k=d/2$ for even $d$. After that we give a proof for the
other values of $k$.   For brevity, write $D = d/2$.
We begin by showing that $c_D=1$.
By \eqref{eq:4.16}  with $j=k = D$,
\[
Z(D,x) = \{ A(D,x), A'(D,x) \} = \{x_D\}.
\]
Thus $A(D,x)=A'(D,x)=x_D$, so that $c_D \in \{0,1\}$.
We cannot have $c_D=0$, otherwise 
$x_D = x$, contradicting the fact that $x$ has degree $d$.
Thus
\begin{equation}\label{eq:4.18}
c_D=1.
\end{equation}
By \eqref{eq:3.2} and \eqref{eq:3.6} with $k=D$,
\[
c_{D+1} + c_{D-1} = 2-2c, \quad c_{D+1}c_{D-1} = (1-c)^2.
\]
Solving this system, we obtain
\begin{equation}\label{eq:4.19}
c_{D+1} = c_{D-1} =1-c.
\end{equation}
By \eqref{eq:4.16}  with $j=D-1$,
\[
\{ x_{D-1}, x_{D+1} \} = Z(D-1,x) = Z(D+1,x),
\]
where the last equality follows from \eqref{eq:4.19}.
This completes the proof of \eqref{eq:4.17} when $k=d/2$.
We proceed to prove \eqref{eq:4.17} under the assumption that $k \ne d/2$.

By \eqref{eq:4.16} with $j=k$,
we have
$x_k \in Z(k,x)$,
so taking $q$-th powers yields
$x_{k+1} \in Z(k, x^q)$.
Therefore, by Lemma 4.2,
\begin{equation}\label{eq:4.20}
x_{k+1} \in \{A_\mu (k-1,x), A_\lambda (k+1,x)\}.
\end{equation}
Suppose for the purpose of contradiction that
$x_{k+1} = A_\mu (k-1,x)$.
Then $x_{k+1}$ lies in the set 
$Z(k-1,x) = \{ x_{k-1}, x_{d-k+1} \}$,
where the equality follows from 
\eqref{eq:4.16} with $j = k-1$.
We cannot have $x_{k+1} = x_{k-1}$, because $x$ 
and its conjugates have degree $d > 2$.
Nor can we have $x_{k+1} = x_{d-k+1}$,
since this would imply that $k=d/2$.
Thus we obtain our desired contradiction.
In view of \eqref{eq:4.20}, it therefore follows that
$x_{k+1} \in Z(k+1,x)$.
To prove \eqref{eq:4.17}, it now suffices to prove
\begin{equation}\label{eq:4.21}
x_{d-k-1} \in Z(k+1,x),
\end{equation}
which is equivalent (via taking the $q$-th power) to
\begin{equation}\label{eq:4.22}
x_{d-k} \in Z(k+1,x^q).
\end{equation}

From Lemma 4.2,
\begin{equation}\label{eq:4.23}
Z(k+1,x^q)=\{A_\mu (k,x), A_\lambda (k+2,x) \}.
\end{equation}
By \eqref{eq:4.16} with $j=k$,
\[
A_\mu (k,x) \in Z(k,x) = \{ x_k, x_{d-k} \}.
\]
If $A_\mu (k,x) = x_{d-k}$, then \eqref{eq:4.22}
holds by \eqref{eq:4.23}, and the proof is complete.
It remains to prove that the alternative
\begin{equation}\label{eq:4.24}
A_\mu (k,x) = x_k
\end{equation}
is impossible.   Thus assume for the purpose of contradiction
that \eqref{eq:4.24} holds.
Then by \eqref{eq:4.23}, 
$x_k \in Z(k+1,x^q)$.
Raise both sides to the $q^{d-1}$-th power to obtain
$x_{k-1} \in Z(k+1, x)$.
By \eqref{eq:4.6} with $j=k-1$, we see that
$x_{k-1} \in Z(k-1, x)$.
It follows that for some choice of the  $\pm$ signs,
\begin{equation*}
\begin{split}
x_{k-1} =&
c_{k-1} + x -2xc_{k-1} \pm 2 \sqrt{c_{k-1} - c_{k-1}^2} \sqrt{x-x^2} \\
=& c_{k+1}+x -2xc_{k+1} \pm 2 \sqrt{c_{k+1} - c_{k+1}^2} \sqrt{x-x^2}.
\end{split}
\end{equation*}
Appealing to the distinctness of the zeros in \eqref{eq:4.10},
we see that this is only possible if $c_{k-1} = c_{k+1}$.
By \eqref{eq:3.2} and \eqref{eq:3.6},
\[
c_{k-1} + c_{k+1} = 2c + 2c_k -4cc_k, \quad c_{k-1}c_{k+1} = (c - c_k)^2.
\]
Thus
\[
0 = (c_{k-1} - c_{k+1})^2 = 
(2c+2c_k-4cc_k)^2 -4(c - c_k)^2 =16(c-c^2)(c_k - c_k^2),
\]
which implies that $c_k \in \{ 0, 1 \}$.
It follows that $Z(k,x)$ is a singleton set.
By \eqref{eq:4.16} with $j=k$, we have
$Z(k,x) = \{ x_k, x_{d-k} \}$,
so that $x_k = x_{d-k}$.
Thus we obtain the contradiction $k = d/2$.
As a result, \eqref{eq:4.24} cannot hold, and the proof is complete.
\end{proof}

The next theorem shows that $d=e$, where
$e$ is the period of $S(c)$.

\begin{thm}\label{Theorem 4.4}
Let $1-s^2 = c = c_1 \in \C$ and let $x$ be a zero of $g_s(y)$ of degree $d >2$
over $\F$.  Then $d=e$.
\end{thm}

\begin{proof}
By Theorem 4.3, $Z(d,x)= \{ x \}$.  Hence $A(d,x) = A'(d,x) =x$,
so that $c_d=0$.   Assume for the purpose of contradiction that
$c_k=0$ for some $k$ with $0 < k < d$.  Then
$A(k,x) = A'(k,x) =x$, so that $Z(k,x)= \{ x \}$.
Again by Theorem 4.3, it follows that $x_k = x_{d - k} = x$, which is not possible.
This contradiction shows that $d$ equals the smallest positive integer $k$
for which $c_k=0$.  Consequently, by \eqref{eq:3.9}, $d=e$.
\end{proof}

Define the monic polynomial $N(y) \in \F[y]$ of degree $d$ by
\begin{equation}\label{eq:4.25}
N(y): =
\begin{cases}
y \prod\limits_{k=1}^{(d-1)/2} (y-c_k)^2, & \ \text{  if  }  \ 2\nmid d \\
(y^2-y) \prod\limits_{k=1}^{(d-2)/2} (y-c_k)^2, & \ \text{  if  }  \ 2\mid d.
\end{cases}
\end{equation}
Note that the coefficients of $N(y)$ can be expressed as polynomials
in $c$ (and hence in $s$) over the integers.  For some small values
of $d$, Corollaries 4.6--4.12 below give explicit formulas
for the coefficients of $N(y)$.

For a zero $x$ of $g_s(y)$, define the monic polynomial $I_x(y)$ 
of degree $d$ by
\begin{equation}\label{eq:4.26}
I(y) = I_x(y): = N(y) - N(x).
\end{equation}
Theorem 4.5 below  shows that these $I_x(y)$ are the 
 irreducible factors of $g_s(y)$.
Thus the monic irreducible factors of $g_s(y)$ 
are all identical except for their constant terms, 
and Theorem 4.5 provides a way of expressing the
coefficients of the nonconstant terms
as polynomials in $c$ over the integers.  
\begin{thm}\label{Theorem 4.5}
Let $1-s^2 = c = c_1 \in \C$ and 
let $x$ be any zero of $g_s(y)$ of degree $d >2$
over $\F$. Then $I_x(y)$ is the monic irreducible polynomial of $x$ over $\F$. 
\end{thm}

\begin{proof}
Theorem 4.3 shows that
\begin{equation}\label{eq:4.27}
\{ A(k,x), A'(k,x) \} = \{x_k, x_{d-k} \}, \quad 0 \le k \le d/2.
\end{equation}
When $0 < k < d/2$, 
we have
$A(k,x) A'(k,x) = (x-c_k)^2$.
When $k=d/2$ with $d$ even, the proof of \eqref{eq:4.18} shows that $c_{k} =1$,
so that $A(k,x) = 1-x$. When $k=0$, we have $c_k=0$, so that $A(k,x) = x$.
Thus \eqref{eq:4.27} yields
\begin{equation}\label{eq:4.28}
(-1)^{d-1} \prod\limits_{k=0}^{d-1} x_k = N(x).
\end{equation}
In particular, $I_x(y) \in \F[y]$, since $N(x)$ 
equals $(-1)^{d-1}$ times the norm of $x$.
As $x$ has degree $d$ and $x$ is a zero of the polynomial
$I_x(y)$ of degree $d$, it follows that
$I_x(y)$ is the monic irreducible polynomial of $x$ over $\F$.

\end{proof}

\begin{cor}\label{Corollary 4.6}
Let $s \in \{ \pm 1/2 \}$ and suppose that $q \equiv \pm 1 \pmod {12}$,
so that $c=1-s^2 = 3/4 \in \C$.  Then each irreducible factor $I(y)$ of $g_s(y)$
has the form
\[
I(y) = y^3 - \frac{3}{2} y^2 + \frac{9}{16} y - m
\]
for some $m \in \F$.
\end{cor}

\begin{proof}
Computing the first four terms of the 
sequence $S(c)$ starting with $c_0$, we find that
$c_0=c_3=0, \ c_1=c_2=3/4$.
Thus $S(c)$ has period $e=3$, so for any zero $x$
of $g_s(y)$, the irreducible factor 
$I_x(y)$ has degree $d=3$.
The result now follows from \eqref{eq:4.26}, since
\[
N(y) = y(y-3/4)^2 = y^3 - \frac{3}{2} y^2 + \frac{9}{16} y.
\]
\end{proof}

\begin{cor}\label{Corollary 4.7}
Suppose that $q \equiv \pm 1 \pmod {8}$, so that $\rho(2) = 1$.
Let \mbox{$s \in \{ \pm 1/\sqrt{2} \},$} 
so that $c=1-s^2 = 1/2 \in \C$.  
Then each irreducible factor $I(y)$ of $g_s(y)$
has the form
\[
y^4 - 2y^3 + \frac{5}{4} y^2 - \frac{1}{4} y - m
\] 
for some $m \in \F$.
\end{cor}

\begin{proof}
We have
$c_0=c_4=0, \ c_1=c_3=1/2, \ c_2 =1$.
Thus $S(c)$ has period $e=4$, so each irreducible factor
$I(y)=N(y) - N(x)$ has degree $d=4$.
The result now follows, as
\[
N(y) = (y^2 - y)(y-1/2)^2 = y^4 - 2y^3 + \frac{5}{4} y^2 - \frac{1}{4} y.
\]
\end{proof}

\begin{cor}\label{Corollary 4.8}
Suppose that $q \equiv \pm 1 \pmod {20}$,
so that $\rho(5)=1$. Set
$c=(5 +\sqrt{5})/8$ for either choice of the square root.   Then $c \in \C$,  
and each irreducible factor $I(y)$ of $g_s(y)$
has the form
\[
y^5 - \frac{5}{2}y^4 + \frac{35}{16} y^3 - \frac{25}{32} y^2 + \frac{25}{256}y - m
\]
for some $m \in \F$.
\end{cor}

\begin{proof}
For either choice of sign of $\sqrt{5} \in \F$,
the condition on $q$ guarantees the existence of a 
primitive tenth root of unity 
$\zeta \in \F[i]$ for which
\[
-(\zeta -\zeta^{-1})^2/4=(5+\sqrt{5})/8.
\]
Thus
$c$ is a square in $\F$.  Moreover, $1-c$ is a square in $\F$ since
$1-c = \Big( (\sqrt{5} -1)/4 \Big)^2$.
This proves that $c \in \C$.
Define  $c':=(5-\sqrt{5})/8$.   We have
$c_0=c_5=0, \ c_1=c_4 = c, \ c_2=c_3=c'$.
As $S(c)$ has period $e=5$, each irreducible factor
$I(y)=N(y) - N(x)$ has degree $d=5$.
The result now follows, as
$N(y) = y(y-c)^2(y-c')^2$.
\end{proof}

\begin{rem}
While the proof above shows that $(5+\sqrt{5})/2$ is a square in $\F$
when $q \equiv \pm 1 \pmod {20}$, it is also true that
$(5+\sqrt{5})/2$ is a non-square in $\F$  when $q \equiv \pm 9 \pmod {20}$.
To verify this, one can first reduce to the case when $q$ is prime,
and then apply the theorem in \cite[p. 257]{WHF}.
\end{rem}

\begin{cor}\label{Corollary 4.9}
Suppose that $q \equiv \pm 1 \pmod {12}$, and let
$s \in \{ \pm \sqrt{3}/2 \}$,
so that $c=1-s^2 = 1/4 \in \C$.  Then each irreducible factor $I(y)$ of $g_s(y)$
has the form
\[
y^6 -3y^5 +\frac{27}{8}y^4 -\frac{7}{4}y^3 + \frac{105}{256} y^2 - \frac{9}{256} y - m
\]
for some $m \in \F$.
\end{cor}

\begin{proof}
We have
$c_0=c_6=0, \ c_1=c_5=1/4, \ c_2 =c_4=3/4, \ c_3 =1$.
Thus $S(c)$ has period $e=6$, so each irreducible factor
$I(y)$ has degree $d=6$.
The result now follows, as
$N(y) = (y^2-y)(y-1/4)^2(y-3/4)^2$.
\end{proof}

\begin{cor}\label{Corollary 4.10}
Suppose that $q \equiv \pm 1 \pmod {16}$,
so that $\rho(2)=1$, and set
$c=(2 + \sqrt{2})/4$ 
for either choice of the square root.
Then $c \in \C$
and each irreducible factor $I(y)$ of $g_s(y)$
has the form
\[
y^8-4y^7+
\frac{13}{2}y^6-\frac{11}{2}y^5+\frac{165}{64}y^4-\frac{21}{32}y^3
+\frac{21}{256}y^2-\frac{1}{256}y -m
\]
for some $m \in \F$.
\end{cor}

\begin{proof}
For either choice of sign of $\sqrt{2} \in \F$,
the condition on $q$ guarantees the existence of a 
primitive sixteenth root of unity $\zeta \in \F[i]$ for which
\[
-(\zeta -\zeta^{-1})^2/4=(2+\sqrt{2})/4.
\]
Thus
$c$ and $c':=1-c$ are both squares in $\F$.  
This proves that $c \in \C$.
We have
\[
c_0=c_8=0, \ c_1 = c_7=c, \ c_2=c_6=1/2, \ c_3= c_5=c', \ c_4 =1.
\]
As $S(c)$ has period $e=8$, each irreducible factor
$I(y)$ has degree $d=8$.
The result now follows, as
$N(y) = (y^2-y)(y-c)^2(y-1/2)^2(y-c')^2$.
\end{proof}

\begin{rem}
While the proof above shows that $(2+\sqrt{2})$ is a square in $\F$
when $q \equiv \pm 1 \pmod {16}$, it is also true that
$(2+\sqrt{2})$ is a non-square in $\F$  when $q \equiv \pm 7 \pmod {16}$.
To verify this, again one can first reduce to the case when $q$ is prime,
and then apply the theorem in \cite[p. 257]{WHF}.
\end{rem}

\begin{cor}\label{Corollary 4.11}
Suppose that $q \equiv \pm 1 \pmod {20}$,
so that $\rho(5)=1$. Set
$c=(3 - \sqrt{5})/8$ for either choice of the square root.   Then $c \in \C$,
and each irreducible factor $I(y)$ of $g_s(y)$
has the form
\begin{equation*}
\begin{split}
&y^{10} - 5y^9 +\frac{85}{8}y^8-\frac{25}{2}y^7+\frac{2275}{256}y^6
-\frac{1001}{256}y^5 +\\
&+ \frac{2145}{2048}y^4 - \frac{165}{1024} y^3 + \frac{825}{65536} y^2
 - \frac{25}{65536}y - m
\end{split}
\end{equation*}
for some $m \in \F$.
\end{cor}

\begin{proof}
The proof of Corollary 4.8 shows that $c \in \C$.
Define  $c':=(3+\sqrt{5})/8$.   We have 
\[
c_0=0, \ c_1 = c, \ c_2=c +1/4, \ c_3= c', \ c_4 =c'+1/4, \ c_5=1,
\]
and $e=d=10$.
The result now follows, as
\[
N(y) = (y^2-y)(y-c)^2(y-c-1/4)^2(y-c')^2(y-c'-1/4)^2.
\]
\end{proof}

\begin{cor}\label{Corollary 4.12}
Suppose that $q \equiv \pm 1 \pmod {24}$,
so that $\rho(2)=\rho(3)=1$. Set
$c=(2 + \sqrt{3})/4$ for either choice of the square root.   Then $c \in \C$,
and each irreducible factor $I(y)$ of $g_s(y)$
has the form
\begin{equation*}
\begin{split}
&y^{12} - 6y^{11}+ \frac{63}{4}y^{10} - \frac{95}{4}y^9 +\frac{2907}{128}y^8
-\frac{459}{32}y^7+
\frac{1547}{256}y^6\\ 
&-\frac{429}{256}y^5 +
 \frac{19305}{65536}y^4 - \frac{1001}{32768} y^3 + \frac{429}{262144} y^2
 - \frac{9}{262144}y - m
\end{split}
\end{equation*}
for some $m \in \F$.
\end{cor}

\begin{proof}
We know $c$ is a square in $\F$ because
$c = (\sqrt{6} + \sqrt{2})^2 /16$.
Similarly,
$1-c=c':=(2  -\sqrt{3})/4$ is a square in $\F$.   
Thus $c \in \C$. We have
\[
c_0=0, \ c_1 = c, \ c_2=1/4, \ c_3= 1/2, \ c_4 =3/4, \ c_5=c', \ c_6=1.
\]
and $e=d=12$.
The result now follows, as
\[
N(y) = (y^2-y)(y-c)^2(y-1/4)^2(y-1/2)^2(y-3/4)^2 (y-c')^2.
\]
\end{proof}

We next give a necessary and sufficient condition on $s$ for the
irreducibility of $g_s(y)$. 
Recall from \eqref{eq:3.13} that $E$ denotes the even integer
$(q-\rho(-1))/2$.

\begin{cor}\label{Corollary 4.13}
The polynomial $g_s(y)$ is irreducible over $\F$ if and only if
$s = (B + B^{-1})/2$ for some element $B  \in \F[i]$ of order $2E$.
Moreover, for such $s$,  $g_s(y)$ is the irreducible polynomial in $\F[y]$
given by 
\begin{equation}\label{eq:4.29}
g_s(y) = \frac{\tau^2}{2} (y^2-y)\prod\limits_{a \in \C} (y-a)^2 + 1-s.
\end{equation}
\end{cor}

\begin{proof}
Since $g_s(y)$ has degree $E$, $g_s(y)$ is irreducible if and only if
the sequence $S(c)$ has period $E$, i.e., if and only if
\[
1-s^2 = c = -(B-B^{-1})^2/4
\]
for some element $B  \in \F[i]$ of order $2E$.
Thus $g_s(y)$ is irreducible if and only if
$s = (B + B^{-1})/2$ for some element $B  \in \F[i]$ of order $2E$.
(Note that if $B$ has order $2E$, then so does $-B$.)
Since $g_s(y)$ has constant term $1-s$,
Theorem 4.5 yields
\[
g_s(y) = \frac{\tau^2}{2} N(y) + (1-s),
\]
so  \eqref{eq:4.29} follows from \eqref{eq:3.15}.
\end{proof}

We remark that the elements $a$ in \eqref{eq:4.29}
can be rapidly calculated, in view of \eqref{eq:3.14}--\eqref{eq:3.15}.
To illustrate Corollary 4.13, 
first consider the case $q=17$ with the choice of primitive root
$B=3$ and the choice  $s = (B+B^{-1})/2 = 13$.   We have in this case
$c=2$,  $\tau = 1$, $d=E = 8$, and
$\C = \{ 2,9,16 \}$.  Thus $2g_s(y)$ equals the irreducible monic polynomial
in $\F[y]$ given by
\[
(y^2-y)(y-2)^2(y-9)^2(y-16)^2 -7.
\]
Next consider the case $q=19$ with the choice $B=4+2i \in \F[i]$ of order $q+1 =20$
and the choice $s = (B+B^{-1})/2 = 4$.  We have in this case
$c=4$, $\tau=2$, $d=E=10$, and $\C=\{ 4, 9, 11, 16 \}$.
Thus $g_s(y)/2$ equals the monic irreducible polynomial
in $\F[y]$ given by
\[
(y^2-y)(y-4)^2(y-9)^2(y-11)^2(y-16)^2 -11.
\]

\section{Norms of the zeros of $g_s(y)$}
Choose an arbitrary element $\zeta \in \F[i]$ of order $2d$,
where $d>2$ divides $E=(q-\rho(-1))/2$.  Define the set
\[
\B_d:=\{(\zeta^j +\zeta^{-j})/2: 0 < j < d, (j,2d)=1\},
\]
and let $\B'_d$ be the set of negatives of the elements in $\B_d$.
If $d$ is even, then $\B_d=\B'_d$, since $\zeta^j +\zeta^{-j}=$
$-(\zeta^{d-j} +\zeta^{j-d})$ and $(d-j,2d)=1$.
However, if $d$ is odd, then $\B_d$ and $\B'_d$ are disjoint.
To see this, suppose otherwise, so that 
$\zeta^k + \zeta^j =-\zeta^{-j-k} (\zeta^k + \zeta^j)$
for some $j, k$ with
$0 < j < k <d$ and $(jk,2d)=1$.
Then either $\zeta^{k-j}$ or $\zeta^{-k-j}$ equals $-1$,
which is impossible since $k-j$ and $-k-j$ are both even.

As $\beta$ runs through the elements of $\F[i]$
of order $2d$, $(\beta + \beta^{-1})/2$ runs twice through the elements of
$\B_d$, and if moreover $d$ is odd,
$-(\beta + \beta^{-1})/2$
runs twice through the elements of $\B'_d$.

In view of \eqref{eq:3.10}, the elements $s \in \B_d \cup \B'_d$
are precisely those $s$ for which $1-s^2=c \in \C$ and the irreducible
factors of $g_s(y)$ and $g_{-s}(y)$ have degree $d$.  
Define the polynomial $G_s(y) \in \F[y]$ of degree $2E$ by
\[
G_s(y):= 2g_s(y)g_{-s}(y).
\]
When $s \in \B_d \cup \B'_d$, 
the irreducible factors $I_x(y):=N(y) - N(x)$ of $G_s(y)$ are polynomials
in $\F[y]$ of degree $d$ 
that are identical except for their constant terms
$-N(x)$.
This raises the following natural question:
How can the 
constant terms corresponding to $g_s(y)$
be distinguished from those corresponding to $g_{-s}(y)$?
For odd $d$, the answer is given by the following residuacity criterion for the
norms $N(x)$ of the zeros $x$ of $G_s(y)$.
\begin{thm}\label{Theorem 5.1}
If  $s \in \B$,
then the norms of the zeros of $g_s(y)$ are nonsquares in $\F$,
and when $d$ is odd, the norms of the 
zeros of $g_{-s}(y)$ are squares in $\F$.
\end{thm}

\begin{proof}
By \eqref{eq:1.3},
\[
G_t(y)/2 = \big(y^n +(1-y)^n\big)^2 - t^2.
\]
With the change of variable $y = (1+u)/2$,
\begin{equation}\label{eq:5.1}
\begin{split}
G_t(y)/2 &= \Big(\Big(\frac{1+u}{2}\Big)^n +
\Big(\frac{1-u}{2}\Big)^n\Big)^2 - t^2  \\
&=\frac{1}{4}f(u^2)^2-t^2=\frac{1}{4}\Big(f((2y-1)^2)^2-4t^2\Big),
\end{split}
\end{equation}
since $4^n=4$ in $\F$.
By \eqref{eq:5.1} and \eqref{eq:1.3},
\begin{equation}\label{eq:5.2}
g_s((2y-1)^2) = G_t(y),
\end{equation}
where $s = 2t^2-1$.

As $s \in \B_d$, we have $s = (\beta + \beta^{-1})/2$
for some $\beta$ of order $2d$.
Let $\delta$ be an element in $\F[i]$ of order $4d$ with $\delta^2 = \beta$.
Then \eqref{eq:5.2} holds with
$t: = (\delta + \delta^{-1})/2$, since $s = 2t^2 -1$.
While each zero  $x$ of $g_s$ has degree $d$ over $\F$,
each zero $v$ of $G_t$ has degree $2d$ over $\F$,
since $t \in \B_{2d}$.
In particular, $v \notin \F[x]$ for such $x,v$.
As $v$ runs through the $2E$ distinct zeros of $G_t$,
$(2v-1)^2$ runs twice through the $E$ distinct zeros of $g_s$.
Thus each zero $x$ of $g_s$  has the form $x = (2v-1)^2$.
It follows that $x$ must be a nonsquare in $\F[x]$, since $2v-1 \notin \F[x]$.
This proves that the norms of the zeros of $g_s(y)$
are nonsquares in $\F$ when $s \in \B_d$.

Finally, let $s \in \B'_d$ with $d$ odd, so that 
$s = -(\beta + \beta^{-1})/2$
for some $\beta$ of order $2d$.
It remains to show that the
norms of the zeros of $g_s(y)$ are squares in $\F$.
Let $j$ denote the odd integer in the set $\{(d+1)/2, (3d+1)/2 \}$.
Then \eqref{eq:5.2} holds with
$t: = (\beta^j + \beta^{-j})/2$, since $s = 2t^2 -1$.
Each zero  $x$ of $g_s$ has degree $d$ over $\F$,
and the same is true about each zero $v$ of $G_t$,
since $t \in \B_d$.
As $v$ runs through the $2E$ distinct zeros of $G_t$,
$(2v-1)^2$ runs twice through the $E$ distinct zeros of $g_s$.
Thus each zero $x$ of $g_s$  has the form $x = (2v-1)^2$.
Since $\F[x] \subset \F[v]$ and both fields have degree $d$ over $\F$,
we must have $2v-1 \in \F[x]$.
Therefore  $x$ is a square in $\F[x]$, which
proves that the norms of the zeros of $g_s(y)$
are squares in $\F$ when $s \in \B'_d$.
\end{proof}

For example, take $d=5$, $q \equiv \pm 1 \pmod {20}$, and $s = (1+\sqrt{5})/4$
in $\F$, for either choice of the square root.  Then $s \in \B_5$, so the
norms of the zeros of $g_s(y)$ are nonsquares
in $\F$, while the norms of the zeros of
$g_{-s}(y)$ are squares in $\F$. 
To illustrate with $q=19$ and $s \in \{12,17\} \subset \B_5$, we have
\begin{equation*}
\begin{split}
&g_{12}(y) = 2(N(y)-3)(N(y)-14), \quad g_{-12}(y) = 2(N(y)-1)(N(y)-16),\\
&g_{17}(y) = 2(N(y)-2)(N(y)-15), \quad g_{-17}(y) = 2(N(y)-6)(N(y)-11).
\end{split}
\end{equation*}
where $3,14,2,15$ are nonsquares modulo 19, while $1,16,6,11$ are squares.
For an example with $d$ even,
take $d=6$, $q=37$, and $s=26$.
Then
\[
2G_s(y)=(N(y)-2)(N(y)-5)(N(y)-14)(N(y)-20)(N(y)-29)(N(y)-32),
\]
where $2,5,14,20,29,32$ are all nonsquares modulo 37.

Finally, we return to the case $d=3$ which motivated this paper.
Define 
\[
G(y): = g_{-1/2}(y) g_{1/2}(y), \quad V = \{ v(v-3/4)^2 : v \in \F \}. 
\]
The monic irreducible cubic factors
$I_x(y)=N(y)-N(x)$ of $G(y)$ are given in 
Corollary 4.6. In the next theorem, we
characterize the set of constant terms of these irreducible factors.
Equivalently,  we characterize $U$, 
where $U$  denotes the set of norms $N(x)$ of the zeros $x$
of $G(y)$.  

\begin{thm}\label{Theorem 5.2}
We have  $U=T$, where $T$ is the complement of $V$ in $\F$. 
\end{thm}

\begin{proof}
Consider the list
\[
L:=\langle m(m-3/4)^2:  m \in \F \rangle,
\]
which has $q$ entries in $\F$.
The entry $0$ occurs twice in $L$ (for $m=0$ and $m=3/4$) and
the entry $1/16$ also occurs twice (for $m=1/4$ and $m=1$).
Solving the equation
\mbox{$w(w-3/4)^2 = m(m-3/4)^2$}
for $w$, we see that of the remaining
entries $m(m-3/4)^2$ in $L$,
those  with $\rho(m(1-m))=-1$ occur once in $L$, and those with
$\rho(m(1-m))=1$ occur thrice.
The number of distinct entries in $L$ is therefore
\[
|V| = 2 +\frac{1}{2}\sum(1-\rho(m(1-m))) + \frac{1}{6}\sum (1+\rho(m(1-m))),
\]
where the sums are over all $m \in \F$ except $m=0,1/4,3/4,1$.
The first term $2$ on the right side cancels out when the sums are 
taken over all $m$.
By \cite[Theorem 2.1.2]{BEW},
\[
\sum \limits_{m \in \F} \rho(m(1-m)) = -\rho(-1),
\]
so that
\[
|V| = (2q + \rho(-1))/3, \quad  |T| = (q-\rho(-1))/3.
\]

Write $E:=(q-\rho(-1))/2$, and let $x$ denote any
of the $2E$ zeros of $G(y)$.
Since the $2E$ zeros are distinct,
$G(y)$ has $2E/3$ distinct factors of the form $I_x(y)$.
In particular, $|U| = 2E/3 = |T|$.
If one of the $2E/3$ values of $N(x) \in U$ were in $V$,
i.e., if $N(x) = N(v)$ for some $v \in \F$, then we'd have the contradiction
that $I_x(y)$ has a linear factor $y-v$.   Thus $U \subset T$.
Since $|U|=|T|$, this proves that $U = T$.
\end{proof}

\end{document}